\documentclass[11pt]{amsart}
\usepackage{setspace}
\usepackage{amsmath, amsfonts, amsthm, amstext, xspace}
\usepackage{amssymb,latexsym}
\usepackage{comment}
\usepackage{enumitem}
\usepackage{xcolor}
\usepackage{hyperref}
\usepackage[capitalise]{cleveref}
\definecolor{seagreen}{RGB}{46,139,87}
\definecolor{maroon}{RGB}{128,0,0}
\definecolor{darkviolet}{RGB}{148,0,211}
\definecolor{twelve}{RGB}{100,100,170}
\definecolor{thirteen}{RGB}{100,150,50}
\definecolor{fourteen}{RGB}{200,0,0}
\definecolor{fifteen}{RGB}{0,200,0}
\definecolor{sixteen}{RGB}{0,0,200}
\definecolor{seventeen}{RGB}{200,0,200}
\definecolor{eighteen}{RGB}{0,200,200}

\usepackage[all]{xy}
\newtheorem{thm}{Theorem}[section]
\newtheorem*{theorem*}{Theorem}
\newtheorem*{conjecture*}{Conjecture}
\newtheorem*{corollary*}{Corollary}
\newtheorem{lem}[thm]{Lemma}
\newtheorem{cor}[thm]{Corollary}

\newtheorem{prop}[thm]{Proposition}

\theoremstyle{definition}
\newtheorem*{exm*}{Example}

\newtheorem{exm}[thm]{Example}

\newtheorem{rem2}[thm]{Remark}

\newcommand{\lra}{\longrightarrow}

\newtheorem{thmx}{Theorem}

\def\c{\mathbb{C}}

\def\q{\mathbb{Q}}
\def\r{\mathbb{R}}

\def\z{\mathbb{Z}}

\def\rk{\operatorname{rk}}
\def\Aut{\operatorname{Aut}}

\usepackage{amssymb}
\usepackage{tikz-cd}

\newcommand{\upi}{\underline{\pi}}
\newcommand{\piGa}{\pi^G_{\alpha}(S^0)}
\newcommand{\piG}{\pi^G}
\newcommand{\ROG}{RO(G)}
\newcommand{\ROPG}{RO^+(G)}

\newcommand{\sm}{\wedge}
\newcommand{\Hom}{\mathrm{Hom}}
\newcommand{\Z}{\mathbb{Z}}
\newcommand{\Q}{\mathbb{Q}}
\newcommand{\bbA}{\mathbb{A}}
\newcommand{\st}{\; |\;}
\newcommand{\tensor}{\otimes}

\author{J.P.C. Greenlees}\address{Mathematics Institute, Zeeman Building, Covengtry CV4, 7AL, UK}\email{john.greenlees@warwick.ac.uk}
\author{J.D. Quigley}\address{Department of Mathematics, Cornell University, Ithaca, NY, USA}\email{jdq27@cornell.edu}
\title{Ranks of $RO(G)$-graded stable homotopy groups of spheres for
  finite groups $G$}

\begin{document}
\maketitle

\input{xypic}

\begin{abstract}
We describe the distribution of infinite groups within the
$RO(G)$-graded stable homotopy groups of spheres for a finite group $G$.
\end{abstract}

\maketitle

\tableofcontents
\section{Introduction}
\subsection{Overview}

In ordinary stable homotopy theory, one of the most basic theorems is
Serre's Finiteness Theorem \cite{Ser53} stating that the $n$-th stable homotopy group of the sphere,
 $\pi_n(S^0)$, is finite for $n>0$. Since
we understand that $\pi_0(S^0)\cong \Z$, this means that rationally the
structure of stable homotopy is very simple, and attention is quickly
focused on torsion.  Equivariantly, it is still true that 
rationalisation is a massive simplification, but the residual
structure in the rationalisation is worth
some attention. 

Let $G$ be a finite group and consider the $RO(G)$-graded stable
homotopy groups of the sphere \cite[Ch. IX]{Alaska}. The purpose of this note is to identify
the crudest feature of these groups: their ranks as abelian
groups. This is a straightforward deduction from well-known results,
but some interesting features emerge by giving a systematic account.

\begin{exm*}
Let $G = C_2$ be the cyclic group of order two. Then
$$RO(C_2) \cong \z\{1, \sigma\},$$
where $1$ is the one-dimensional trivial representation and $\sigma$ is the one-dimensional
sign representation. Computations of Araki--Iriye \cite{AI82} show that $\pi_\alpha^{C_2}(S^0)$ is 
infinite if
$$\alpha \in \z \{ 2(1 - \sigma) \} \cup \z \{\sigma\}.$$
Our results recover this observation, and show that these are the \emph{only} degrees for which $\pi_\alpha^{C_2}(S^0)$ is infinite. 
\end{exm*}

Using rational equivariant stable homotopy theory, we prove the following:

\begin{thmx}[{\cref{cor:piafinite}}]\label{Thmx:A}
Let $G$ be a finite group and $\alpha \in RO(G)$. Then
  $$\piGa \otimes \q =[S^{\alpha},S^0]^G \otimes \q =\prod_{(H)}\Hom_{W_G(H)}(\pi_0(S^{\alpha^H}), \Q),$$
where the product is taken over conjugacy classes of subgroups $H \leq G$.
  Thus $\piGa \otimes \q$ is a rational vector space of dimension $r_{\alpha}$, where 
  $$r_{\alpha}=|\{ (H)\st \alpha^H=0 \mbox{ and } W_G(H) \mbox{ acts trivially
    on } \pi_0(S^{\alpha^H}) \}|.$$
\end{thmx}

We lay the groundwork for applying this theorem in \cref{Sec:Orientation} and \cref{Sec:Geometry}. We then compute the ranks of the
$RO(G)$-graded stable homotopy groups of spheres for various $G$ in \cref{Sec:Examples}.

In \cref{Sec:Variations}, we discuss two natural variations where the same techniques give information. 
Since the sphere is rationally an
Eilenberg-MacLane spectrum for the Burnside Mackey functor,
$S^0\simeq_{\Q} H\bbA$, we may view our methods as a calculation of the
rationalization of  $H^{\star}_G(S^0; \bbA)$ (where $\star$ denotes
$RO(G)$-grading). The same methods apply to
give a calculation of the rationalization of  $H^{\star}_G(S^0;
M)$ for any  Mackey functor $M$. For the second variation,  we may
consider the Picard-graded stable homotopy groups of spheres:
invertible objects are again characterised in terms of orientations
and dimension functions (see \cite{FLM}). 

Finally, we note that our results provide a basis for understanding other large-scale phenomena in the $RO(G)$-graded stable homotopy groups of spheres. For example, Iriye \cite{Iri82} showed that Nishida's nilpotence theorem \cite{Nis73} holds equivariantly: an element $\pi_\star^G(S^0)$ is torsion if and only if it is nilpotent. \cref{Thmx:A} therefore explicitly describes the regions of $\pi_\star^G(S^0)$ in which elements can be nilpotent and non-nilpotent. 

\subsection{Finite generation}
For most of the paper we will work rationally, but we would like to
draw conclusions about the integral situation. For completeness we
include the proofs of the basic finiteness statements that permit this
deduction. 
\begin{lem}\label{lem:finite}
For any $\alpha \in RO(G)$, the sphere $S^{\alpha}$ is a finite $G$-cell spectrum.
\end{lem}

\begin{proof}
For an actual representation $V$, the sphere $S^V$ is a smooth compact
manifold and hence admits the structure of a finite
$G$-CW-complex. By exactness of Spanier--Whitehead duality, $DS^V\simeq
S^{-V}$ is also a finite $G$-CW spectrum (since by the Wirthm\"uller
isomorphism $DG/H_+\simeq G/H_+$). Now if $\alpha =V-W$,
$S^{\alpha}\simeq S^V\sm S^{-W}$, so the result follows. 
  \end{proof}

The following consequence fails for infinite compact Lie groups. 
\begin{lem}
For any $\alpha \in \ROG$ the abelian group $\piGa$ is finitely
generated. Consequently 
$\piGa$ is finite if and only if $\piGa\tensor \Q=0$. 
\end{lem}

\begin{proof}
From the Segal--tom Dieck splitting theorem \cite{tD75}, we see that $\piG_n(S^0)$
is a finitely generated abelian group. By \cref{lem:finite}, it
follows that $\piGa$ is finitely generated. 
\end{proof}

\cref{Thmx:A} describes the $RO(G)$-graded rational
homotopy groups of the sphere. By the previous lemma, this determines precisely those 
degrees $\alpha \in RO(G) $ for which $\piGa$ is finite.

\subsection{Conventions}
Henceforth everything is rational. We write $G$ for a finite group, $H$ for a subgroup of $G$, and $W_G(H) = N_G(H)/H$ for the Weyl group of $H$. We use $*$ to denote $\z$-graded groups and $\star$ to denote $RO(G)$-graded groups. If $V \in RO(G)$, then $|V|$ denotes its (virtual) dimension. 

\subsection{Acknowledgements}
The second author thanks Eva Belmont, Bert Guillou, Dan Isaksen, and Mingcong Zeng for helpful discussions related to this work. The authors also thank William Balderrama for discussions concerning nilpotence, and especially for pointing out \cite{Iri82}, which answered a question in a previous version. 

\section{Rational stable homotopy}\label{Sec:Rational}
For finite groups, it is easy to give a complete model of rational
$G$-spectra \cite[App. A]{GMTate}. We do not need the full strength
of this description, so we describe what we want in a convenient form.

First, note that for any $X$ and $Y$, passage to geometric fixed points
gives a map
$$\Phi^H: [X,Y]^G\lra [\Phi^HX, \Phi^HY]. $$
The codomain admits an action of the Weyl group $W_G(H)$ by conjugation, and $\Phi^H$ 
takes values in the $W_G(H)$-equivariant maps.

\begin{thm}\label{Thm:RationalMaps}
  If $Y$ is rational, the maps $\Phi^H$ give an isomorphism
  $$[X,Y]^G_*=\bigoplus_{(H)}H^0(W_G(H); [\Phi^HX, \Phi^HY]_*), $$
where the sum is taken over conjugacy classes of subgroups $H \leq G$.
  Furthermore, passage to homotopy groups gives isomorphisms
  $$H^0(W_G(H); [\Phi^HX, \Phi^HY]_*)=\Hom_{W_G(H)}(\pi_*(\Phi^HX), \pi_*(\Phi^HY)).$$
  \end{thm}

\begin{proof}
Filtering $EG_+$ by skeleta gives a spectral sequence
$$H^*(G; [X,Y]_*)\Rightarrow [EG_+\sm X, Y]^G_* $$
for (integral) stable maps. 
When $Y$ is rational, this collapses to an isomorphism
$$H^0(G; [X,Y]_*)=[EG_+\sm X, Y]^G_*$$
Combining this with the splitting $S^0\simeq \bigvee_{(H)} e_HS^0 $
we obtain the first stated isomorphism. The second comes 
from the classical version of Serre's Theorem \cite{Ser53}.
\end{proof}

\begin{rem2}
 There is an alternative approach, via \cite{GMTate}. First, we observe that 
 $X\simeq \prod_n\Sigma^n H\upi^G_n(X)$, and then use the
 fact that all rational Mackey functors are projective and injective to deduce
 $$[X,Y]^G\stackrel{\cong}\lra \prod_n\Hom( \upi^G_n(X), \upi^G_n(Y)). $$
Now we use the structure of Mackey functors to deduce
$$\Hom( \upi^G_n(X), \upi^G_n(Y)) \cong \prod_{(H)}\Hom_{W_G(H)}(\pi_n(\Phi^HX), \pi_n(\Phi^HY)),$$
as claimed.
\end{rem2}

Since $G$ acts trivially on $S^0$, $W_G(H)$ acts trvially on $\pi_0(S^0)=\Q$. We then have the following
consequence of \cref{Thm:RationalMaps}:

\begin{thm}
  \label{cor:piafinite}
Let $G$ be a finite group and $\alpha \in RO(G)$. Then
  $$\piGa=[S^{\alpha},S^0]^G=\prod_{(H)}\Hom_{W_G(H)}(\pi_0(S^{\alpha^H}), \Q),$$
where the product is taken over conjugacy classes of subgroups $H \leq G$.
  Thus $\piGa$ is a rational vector space of dimension $r_{\alpha}$, where 
  $$r_{\alpha}=|\{ (H)\st \alpha^H=0 \mbox{ and } W_G(H) \mbox{ acts trivially
    on } \pi_0(S^{\alpha^H}) \}|.$$
  \end{thm}

  \section{The orientation character}\label{Sec:Orientation}
  \newcommand{\aut}{\mathrm{Aut}}
  \newcommand{\mutwo}{\mu_2}
For any real representation $V$ the group $G$ acts on $H_{|V|}(S^V)$,
giving a homomorphism
$$o_V: G\lra \aut(\Z)=\mutwo.$$
We view this as an element $o_V\in H^1(G;\mutwo)$. In view of the
K\"unneth isomorphism
$$H_{|V|}(S^V)\tensor  H_{|W|}(S^W)
\stackrel{\cong}\lra H_{|V+W|}(S^{V\oplus W}),$$
this gives a homomorphism
$$o: \ROG \lra H^1(G; \mutwo). $$
Elements of the kernel $\ROPG$ of $o$ are {\em orientable} virtual
representations. 

\begin{exm}
Clearly $2\alpha$ is orientable for any $\alpha$. More generally,
the image of any complex representation is orientable, as is any element
in the image of $RSO(G)\lra RO(G)$.  
\end{exm}

\begin{rem2}
It is clear that an orientable
representation $\rho: G\lra O(n)$ is one that takes values in
representations of determinant 1, so that it comes  from $RSO(G)$.
However, this is not true of virtual
representations. For example,  if $G=\Sigma_3$, then  $V-\sigma - 1$  is
orientable (where $V$ is the reduced regular representation and
$\sigma$ is the sign representation). However, only even multiples of
$\sigma$ or $V$ come from $RSO(\Sigma_3)$. 
\end{rem2}
   
\section{Geometry of the ranks of the $RO(G)$-graded stable stems}\label{Sec:Geometry}
To make the answer in \cref{cor:piafinite} explicit there are
now two ingredients: (a) the dimension of the fixed
points and (b) the orientations.

\subsection{Virtual representations of fixed point dimension zero}
If we list the simple real representations $S_1, S_2, \ldots, S_r$ of $G$, we
may identify $\ROG=\Z^r$. Now,  for each subgroup $H \leq G$ we have a dimension vector
$$d_H=(\dim (S_1^H), \ldots , \dim(S_r^H)), $$
and the space of virtual representations $\alpha$ with
$\alpha^H=0$ is  
$$N_H=\{ x\st x\cdot d_H=0\}, $$
which is isomorphic to $\Z^{r-1}$ as an abelian group.
The only $\alpha$ for which $\piGa$ can be infinite are those lying in
some $N_H$, and the maximum rank of $\piGa$ is the number of conjugacy
classes of $H$ with $\alpha \in N_H$.

When $H=G$, the Weyl group $W_G(H)$ is trivial, and we immediately
draw a useful conclusion.

\begin{cor}
If $V$ is a virtual representation with  $V^G=0$ then 
$$\rk \pi_V^G(S^0) \geq 1$$
\end{cor}

\begin{rem2}
  One special case is when $V$ is a multiple of the reduced regular
  representation $\bar{\rho}$. This was observed  to the
  second author by Bert Guillou, who noted that it follows from the
  fact that $\Phi^G(S^0) \simeq S^0$ and that geometric fixed points
  are given by inverting the Euler class of the reduced regular
  representation.

  On this same theme, if $V$ is a representation with $V^G=0$, the
  inclusion of the origin gives a map $a_V: S^0\lra S^V$ whose
  $G$-fixed points generates $\pi_0(S^0)$.  The element $a_V$ is thus of
  infinite order in $\pi^G_{-V}(S^0)$.  The $G$-component of the map
  $a_V$ will not usually be invertible integrally. However, by \cref{Thm:RationalMaps}, there is a rational map $a_V'\in
  \pi_V^G(S^0)$ whose $G$-component is the inverse of $a_V$. The problem of finding
  the smallest positive multiple of $a'_V$ that is integral is of
  considerable interest; the case of the group of order 2 was studied
  classically by Landweber \cite{Lan69}, but is now best treated using
  motivic homotopy theory \cite{BGI21, GI20}. For the group of odd prime order
  $p$, it was studied by Iriye \cite{Iri89}. 
\end{rem2}

\subsection{Orientability}
If $W_G(H)$ is of odd order, then all the
gradings in $N_H$ give infinite groups.  In general, on each such null
space $N_H$ we have an orientation
$$o_H: N_H\lra H^1(W_G(H); \mutwo). $$
As noted above, the kernel $N_H^+$ contains all even vectors of $N_H$ and the
image of all complex representations.

The set of $\alpha$ for which $\piGa$ is infinite is $\bigcup_H
N_H^+$.  The rank $r_{\alpha}$ of $\piGa$ is the number of conjugacy classes
$H$ with $\alpha \in N_H^+$.

\subsection{Bases}
If we choose a subgroup $H$ giving an associated fixed point vector
$d_H$, we note that the component of the trivial representation $S_1$
is always 1, so that $N_H$ has basis $S_2-d_H(2), S_3-d_H(3), \ldots,
S_r-d_H(r)$.
The orientation $o_H$ is thus described by the 
homomorphisms 
$$o_H(2), o_H(3), \ldots, o_H(r): W_G(H)\lra \mu_2$$
where $o_H(i)=o_H(S_i-d_H(i))$.  Since $W_G(H)$ always acts trivially on the trivial 
representation, the orientation $o_H(S_i-d_H(i))=o_H(S_i)$, 
and  $o_H(i)$ is the determinant of $S_i^H$. Since $o_H$ is
a homomorphism, this determines its values throughout. All the
homomorphisms factor through the largest elementary abelian 2-quotient $E_2(H)$
of $W_G(H)$ (i.e., we factor out commutators and squares). 

\section{The two variations}\label{Sec:Variations}
In effect, our calculation in \cref{cor:piafinite} was of
$$\piGa \tensor \Q =[S^\alpha, S^0]^G\tensor \Q=[S^\alpha,  
H\bbA]^G\tensor \Q=H^0_G(S^\alpha; \bbA)\tensor \Q. $$ 
We point out that the same methods allow us to calculate
$$[S^\alpha, 
HM]^G\tensor \Q=H^0_G(S^\alpha; M)\tensor \Q $$
for any invertible spectrum $S^\alpha$ and rational $G$-Mackey functor $M$. Indeed we still have 
$$[S^\alpha, 
HM]^G\tensor \Q=\prod_{(H)}\Hom_{W_G(H)}(H_0(S^{\alpha^H}), M^{e
  H}), $$
where $M$ corresponds to $\{ M^{e H}\}_H$ under the equivalence
$$\mbox{$G$-MackeyFunctors$/\Q$}\simeq \prod_{(H)}\mbox{$\Q W_G(H)$-modules}.$$
More explicitly, $M^{e H}=M(G/H)/(\mathrm{proper \ transfers})$.
In other words, 
$$\rk [S^\alpha, HM]^G \otimes \q = \sum_{(H)} z_H \cdot m(\alpha , H),$$
where $z_H=1$ if $\alpha^H=0$ and $z_H=0$ otherwise, and  where $m(\alpha
,H)$ is the multiplicity of the simple $\Q W_G(H)$-representation
$H_0(S^{\alpha^H})$ in $M^{e H}$.

The only $M^{eH}$ which can possibly give infinite groups are those with summands
coming from a homomorphism $W_G(H)\lra \mu_2$.  Since the sphere
corresponds to the Burnside Mackey functor $\bbA$ with $\bbA^{eH}=\Q$
(with trivial action), it has almost as many $RO(G)$-gradings which are
infinite as is possible.

\section{Examples}\label{Sec:Examples}

We conclude by explicitly calculating the ranks of the $RO(G)$-graded stable homotopy groups of spheres for groups $G$ with small subgroup lattices. 

\subsection{Cyclic group of order two}

We have
$$RO(C_2) \cong \z\{1, \sigma\}$$
where $1$ is the ($1$-dimensional) trivial representation and $\sigma$ is the sign representation. Then
$$N_e \cong \z \{ 1 - \sigma\}, \quad N_{C_2} \cong \z\{\sigma\}.$$
Since $W_{C_2}(C_2) = e$, we have
$$N_{C_2}^+ = N_{C_2} \cong \z\{\sigma\}.$$
On the other hand, $W_{C_2}(e) \cong C_2/e \cong C_2$ acts by $(-1)$ on $1-\sigma$, so
$$N_e^+ \cong \z\{2(1-\sigma)\}.$$
Each representation $V \in N_{C_2}^+ \cup N_e^+$ satisfies $\rk \pi_V^{C_2}(S^0) \geq 1$. Since $N_{C_2}^+ \cap N_e^+ = \{0\}$, we also have $\rk \pi_0^{C_2}(S^0) = 2$. Altogether, we find:

\begin{prop}
We have
\[
\rk \pi_V^{C_2}(S^0) =
\begin{cases}
	2 \quad & \text{ if } V = 0, \\
	1 \quad & \text{ if } V \in (\z\{\sigma\} \cup \z\{2(1-\sigma)\}) \setminus \{0\}, \\
	0 \quad & \text{ otherwise.}
\end{cases}
\]
\end{prop}

\begin{figure}
\includegraphics[trim=0 40 0 0, clip]{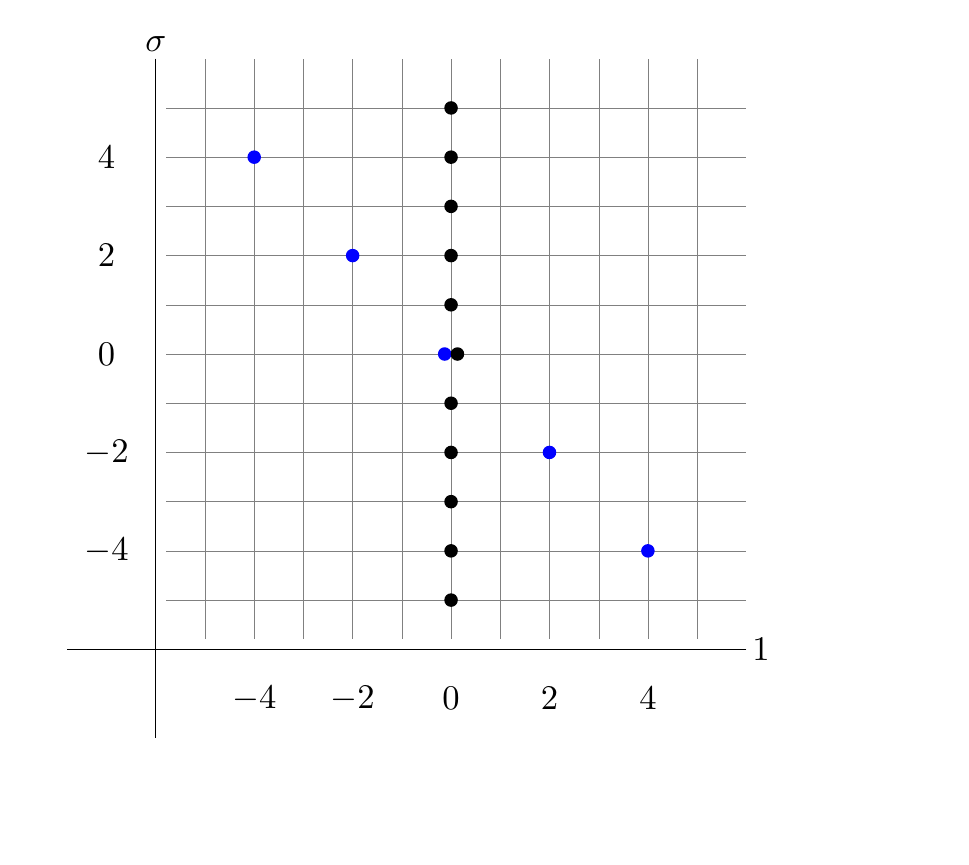}
\caption{Degrees in $RO(C_2)$ where $\pi_V^{C_2}(S^0)$ has infinite rank. A black $\bullet$ indicates a copy of $\z$ arising from $N_{C_2}^+$ and a blue $\bullet$ indicates a copy of $\z$ arising from $N_e^+$.}
\end{figure}

\begin{rem2}
The fact that $\pi_V^{C_2}(S^0)$ is infinite for $V \in \z\{\sigma\} \cup \z\{2(1-\sigma)\}$ appears in \cite[Thm. 7.6]{AI82}. A proof that these are the only degrees for which $\pi_V^{C_2}(S^0)$ is infinite using the $C_2$-equivariant Adams spectral sequence was communicated to the second author by Bert Guillou and Dan Isaksen. 
\end{rem2}

\subsection{Cyclic group of odd prime order}

Let $q = \frac{p-1}{2}$. We have
$$RO(C_p) \cong \z\{ 1, \phi_1,\ldots,\phi_{q}\},$$
where $1$ is the $1$-dimensional trivial representation and $\phi_t: C_p \to \Aut(\r^2) \cong \Aut(\c)$ sends the generator of $C_p$ to $\cdot e^{2\pi i t/p}$. Then
$$N_e \cong \z \{ 2 - \phi_1, \ldots, 2 - \phi_q\}, \quad N_{C_p} \cong \z \{\phi_1,\ldots,\phi_q\}.$$
Since $W_{C_p}(e) \cong C_p$ and $W_{C_p}(C_p) = e$ necessarily act trivially on $\z$, we have
$$N_e^+ = N_e, \quad N_{C_p}^+ \cong N_{C_p}.$$
Finally, we have
$$N_e^+ \cap N_{C_p}^+ \cong \z\{\phi_1 - \phi_2, \ldots, \phi_1 - \phi_q\}.$$

\begin{prop}
We have
\[
\rk \pi_V^{C_p}(S^0) = 
\begin{cases}
	2 \quad & \text{ if } V \in \z \{ \phi_1 - \phi_2, \ldots, \phi_1 - \phi_q\}, \\
	1 \quad & \text{ if } V \in (\z\{2 - \phi_1, \ldots, 2 - \phi_q\} \cup \z \{\phi_1,\ldots,\phi_q\}) \setminus \z\{\phi_1-\phi_2,\ldots,\phi_1-\phi_q\}, \\
	0 \quad & \text{ otherwise.}
\end{cases}
\]
\end{prop}

\begin{figure}
\includegraphics[trim=0 40 0 0, clip]{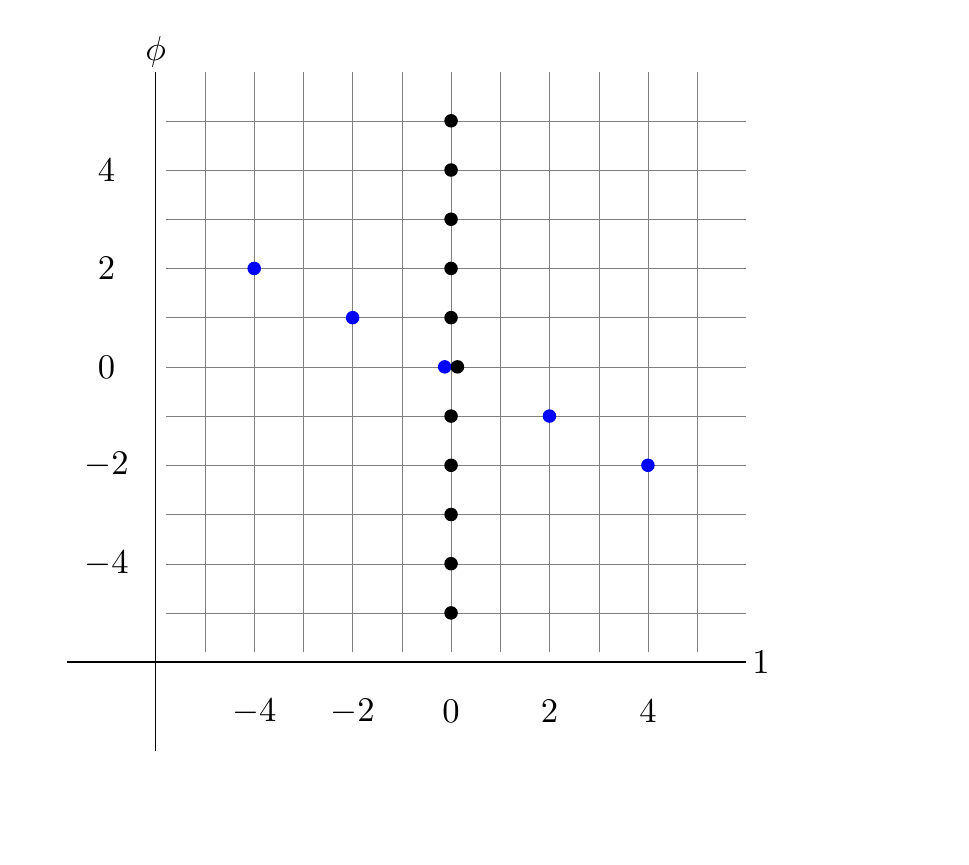}
\caption{Degrees in $RO(C_3)$ where $\pi_V^{C_3}(S^0)$ has infinite rank. A black $\bullet$ indicates a copy of $\z$ arising from $N_{C_3}^+$ and a blue $\bullet$ indicates a copy of $\z$ arising from $N_e^+$.}
\end{figure}

\begin{rem2}
  We note that $\phi_1, \ldots, \phi_q$ have similar behaviour.
  Thus we are considering $\Z \oplus
N_G$  and the same picture as for $C_2$, but now the vertical line
represents $N_G$ and the antidiagonal represents another subspace
isomorphic to $N_G$. The origin now consists of $\{\Sigma_i a_i \phi_i \st
\Sigma_i a_i=0\}$, perhaps to be thought of as the kernel of $N_G\lra
\Z$, and isomorphic to $\Z^{q-1}$. 
  \end{rem2}
\subsection{Cyclic groups of odd prime power order}

Let $q = \frac{p^n-1}{2}$. Then
$$RO(C_{p^n}) \cong \z \{ 1, \phi_1, \ldots, \phi_q\} \cong \z^{q+1}.$$
For all $0 \leq m \leq n$, we have
$$N_{C_{p^m}}^+  = N_{C_{p^m}} \cong \z\{ 2 - \phi_i : p^m \mid i\} \oplus \z\{\phi_j : p^m \nmid j\}.$$
Indeed, let $\gamma$ denote a generator of $C_{p^n}$, so $\gamma^{p^{n-m}}$ is a generator for $C_{p^m}$. Since $\lambda_i: \gamma \mapsto \cdot e^{\frac{2\pi i}{p^n}}$, 
$$\lambda_i: \gamma^{p^{n-m}} \mapsto \cdot e^{\frac{2\pi i p^{n-m}}{p^n}} = \cdot e^{\frac{2 \pi i}{p^m}}.$$
Therefore $\lambda_i$ pulls back to a trivial $C_{p^m}$-representation if and only if $p^m \mid i$. 

Describing the intersections of these subspaces gets complicated quickly. For example, if $0 \leq k < m \leq n$, then
$$N_{C_{p^k}}^+ \cap N_{C_{p^m}}^+ \cong  \z\{2 - \phi_i : p^m \mid i\} \oplus \z\{\phi_j : p^k \nmid i\} \oplus \z\{\phi_{p^k}-\phi_\ell : \ell > p^k, \ p^k \mid \ell, \ p^m \nmid \ell\}.$$
Here, we use that $p^m \mid i$ implies $p^k \mid i$, and similarly, $p^k \nmid j$ implies $p^m \nmid j$.

\subsection{Klein four group}

Let $K = C_2 \times C_2 = \{e, i, j,k\}$. We have
$$RO(K) \cong \z \{ 1, \sigma_i, \sigma_j, \sigma_k\},$$
where $1$ is the $1$-dimensional trivial representation, $\sigma_i$ is the $1$-dimensional representation on which $e$ and $i$ act trivially and $j$ and $k$ act by $(-1)$, and similarly for $\sigma_j$ and $\sigma_k$. Then
\begin{align*}
N_e &\cong \z \{ 1 - \sigma_i, 1-\sigma_j, 1-\sigma_k\}, \\
N_{\langle i \rangle} &\cong \z \{ 1 - \sigma_i, \sigma_j, \sigma_k\}, \\
N_{\langle j \rangle} &\cong \z\{ \sigma_i, 1 - \sigma_j, \sigma_k\}, \\
N_{\langle k \rangle} &\cong \z\{ \sigma_i, \sigma_j, 1 - \sigma_k\}, \\
N_{K} &\cong \z\{\sigma_i, \sigma_j, \sigma_k\}.
\end{align*}
The Weyl group $W_K(e) \cong K$ acts nontrivially on $\sigma_i$, $\sigma_j$, and $\sigma_k$, so we have
$$N_e^+ \cong \z\{2(1-\sigma_i), \sigma_i - \sigma_j, \sigma_i-\sigma_k\}.$$
The Weyl group $W_K(\langle i \rangle) \cong K / \langle i \rangle \cong \langle j \rangle \cong \langle k \rangle$ acts nontrivially on $\sigma_i$ but trivially on $\sigma_j$ and $\sigma_k$, so we have
$$N_{\langle i \rangle}^+ \cong \z \{ 2(1 - \sigma_i), \sigma_j, \sigma_k\},$$
and similarly,
$$N_{\langle j \rangle}^+ \cong \z\{ \sigma_i, 2(1-\sigma_j), \sigma_k\},$$
$$N_{\langle k \rangle}^+ \cong \z\{ \sigma_i, \sigma_j, 2(1-\sigma_k)\}.$$
Finally, since $W_K(K) \cong e$ must act trivially on $\z$, we have
$$N_K^+ \cong N_K \cong \z \{ \sigma_i, \sigma_j, \sigma_k\}.$$

To determine the ranks of $\piGa$, we now compute intersections. In the following, we let $a \in \{i, j, k\}$, $a' \in \{i,j,k\} \setminus \{a\}$, and $a'' \in \{i,j,k\} \setminus \{a, a'\}.$ Then we have
\begin{align*}
N_e^+ \cap N_{\langle a \rangle}^+ &\cong \z\{2(1-\sigma_a), \sigma_a - \sigma_{a'} \}, \\
N_e^+ \cap N_K^+ &\cong \z\{\sigma_i - \sigma_j, \sigma_i - \sigma_k\}, \\
N_{\langle a \rangle}^+ \cap N_{\langle a' \rangle}^+ &\cong \z \{ 2(1 - \sigma_a - \sigma_{a'}), \sigma_{a''} \}, \\
N_{\langle a \rangle}^+ \cap N_K^+ &\cong \z \{ \sigma_{a'}, \sigma_{a''} \},\\
N_e^+ \cap N_{\langle a \rangle}^+ \cap N_{\langle a' \rangle}^+ &\cong \{0\},\\
N_{\langle a \rangle}^+ \cap N_{\langle a'' \rangle}^+ \cap N_K^+ &\cong \z\{\sigma_{a'}\},\\
N_e^+ \cap N_{\langle a \rangle}^+ \cap N_K^+ &\cong \z \{\sigma_a - \sigma_{a'} \},\\
N_{\langle a \rangle} \cap N_{\langle a' \rangle} \cap N_{\langle a'' \rangle} &\cong \z \{ 2(1 - \sigma_a - \sigma_{a'} - \sigma_{a''}) \},
\end{align*}
and all $4$- and $5$-fold intersections are $\{0\}$.  

\begin{prop}
With $a, a', a''$ as above, we have
\[
\rk \pi_V^K(S^0) = 
\begin{cases}
	5 \quad & \text{ if } V = 0, \\
	3 \quad & \text{ if } V \in (\z\{\sigma_a - \sigma_{a'}\} \cup \z\{\sigma_{a'}\} \cup \z\{2(1 - \sigma_a - \sigma_{a'} - \sigma_{a''})\}) \setminus \{0\}, \\
	2 \quad & \text{ if } V \in (\bigcup_{H \neq H'} N_H^+ \cap N_{H'}^+ ) \setminus (\bigcup_{H \neq H' \neq H'' \neq H} N_H^+ \cap N_{H'}^+ \cap N_{H''}^+), \\
	1 \quad & \text{ if } V \in (\bigcup_H N_H^+) \setminus (\bigcup_{H \neq H'} N_H^+ \cap N_{H'}^+), \\
	0 \quad & \text{ otherwise.}
\end{cases}
\]
\end{prop}

\subsection{Dihedral groups of order $2p$, $p$ odd}

We have
$$RO(D_{2p}) \cong \z \{ 1, \sigma, \phi_1, \ldots, \phi_q\},$$
where $1$ is the trivial representation, $\sigma$ is the sign representation, and $\phi_t : D_{2p} \to \Aut(\r^2) \cong \Aut(\c)$ sends the generator of $C_p \subseteq D_{2p}$ to $\cdot e^{2\pi i t / p}$ and the generator of $C_2 \subseteq D_{2p}$ to reflection across the real axis. Then
\begin{align*}
N_e &\cong \z\{1 - \sigma, 2 - \phi_1, \ldots, 2 - \phi_q\}, \\
N_{C_2} &\cong \z \{ 1 - \sigma, 1 - \phi_1, \ldots, 1 - \phi_q\}, \\
N_{C_p} &\cong \z \{ 1 - \sigma, \phi_1, \ldots, \phi_q\}, \\
N_{D_{2p}} &\cong \z \{ \sigma, \phi_1, \ldots, \phi_q\}.
\end{align*}
Since $W_{D_{2p}}(C_2) \cong e \cong W_{D_{2p}}(D_{2p})$, we have $N_{C_2}^+ = N_{C_2}$ and $N_{D_{2p}}^+ = N_{D_{2p}}$. On the other hand, $W_{D_{2p}}(e) \cong D_{2p}$ and $W_{D_{2p}}(C_p) \cong C_2$, so 
\begin{align*}
N_e^+ &\cong \z \{ 2(1 - \sigma), 2(2 - \phi_1), \phi_1-\phi_2, \ldots, \phi_1 - \phi_q\}, \\
N_{C_p}^+ &\cong \z \{2(1 - \sigma), 2\phi_1, \phi_1-\phi_2,\ldots,\phi_1-\phi_q\}.
\end{align*}

We now compute intersections:
\begin{align*}
N_e^+ \cap N_{C_2}^+ &\cong \z \{2(1-\sigma), \phi_1-\phi_2, \ldots, \phi_1 - \phi_q\}, \\
N_e^+ \cap N_{C_p}^+ &\cong \z \{2(1-\sigma), \phi_1-\phi_2, \ldots, \phi_1 - \phi_q\}, \\
N_e^+ \cap N_{D_{2p}}^+ &\cong \z \{4\sigma-2\phi_1, \phi_1-\phi_2, \ldots, \phi_1 - \phi_q\}, \\
N_{C_2}^+ \cap N_{C_p}^+ &\cong \z \{ 2(1 - \sigma), \phi_1-\phi_2, \ldots, \phi_1 - \phi_q\}, \\
N_{C_2}^+ \cap N_{D_{2p}}^+ &\cong \z \{ \sigma - \phi_1, \phi_1-\phi_2, \ldots, \phi_1 - \phi_q\},  \\
N_{C_p}^+ \cap N_{D_{2p}}^+ &\cong \z \{ \phi_1, \dots, \phi_q\}, \\
N_e^+ \cap N_{C_2}^+ \cap N_{C_p}^+ &\cong \z \{ 2(1 - \sigma), \phi_1-\phi_2, \ldots, \phi_1 - \phi_q\}, \\
N_e^+ \cap N_{C_2}^+ \cap N_{D_{2p}}^+ &\cong \z\{ \phi_1-\phi_2, \ldots, \phi_1 - \phi_q \}, \\
N_e^+ \cap N_{C_p}^+ \cap N_{D_{2p}}^+ &\cong \z\{ \phi_1-\phi_2, \ldots, \phi_1 - \phi_q\}, \\
N_{C_2}^+ \cap N_{C_p}^+ \cap N_{D_{2p}}^+ &\cong \z \{ \phi_1-\phi_2, \ldots, \phi_1 - \phi_q\}, \\
N_e^+ \cap N_{C_2}^+ \cap N_{C_p}^+ \cap N_{D_{2p}}^+ &\cong \z \{ \phi_1-\phi_2, \ldots, \phi_1 - \phi_q\}.
\end{align*}

\begin{prop}
We have
\[
\rk \pi_V^{D_{2p}}(S^0) = 
\begin{cases}
4 \quad & \text{ if } V \in \z \{ \phi_1-\phi_2, \ldots, \phi_1 - \phi_q\}, \\
2 \quad & \text{ if } V \in (\bigcup_{H \neq H'} N_H^+ \cap N_{H'}^+ ) \setminus (\bigcup_{H \neq H' \neq H'' \neq H} N_H^+ \cap N_{H'}^+ \cap N_{H''}^+), \\
1 \quad & \text{ if } V \in (\bigcup_H N_H^+) \setminus (\bigcup_{H \neq H'} N_H^+ \cap N_{H'}^+), \\
0 \quad & \text{ otherwise.}
\end{cases}
\]
\end{prop}

\subsection{Quaternion group}

Let $Q = Q_8$ denote the quaternion group of order $8$. We have
$$RO(Q) = \z \{ 1, \sigma_i, \sigma_j, \sigma_k, h\},$$
where $1, \sigma_i, \sigma_j, \sigma_k$ are the pullbacks of the $K$-representations of the same name along the quotient map $Q \to Q/C_2 \cong K$, and $h$ is the unique irreducible $4$-dimensional representation of $Q$. Then with $a$, $a'$, and $a''$ as in our analysis of $K$, 
\begin{align*}
N_e &\cong \z\{1 - \sigma_i, 1-\sigma_j, 1-\sigma_k, 4 1 - h\}, \\
N_{C_2} &\cong \z\{ 1 - \sigma_i, 1 - \sigma_j, 1-\sigma_k, h\}, \\
N_{\langle a \rangle} &\cong \z \{ 1 - \sigma_a, \sigma_{a'}, \sigma_{a''}, h\}, \\
N_Q &\cong \z\{ \sigma_i, \sigma_j, \sigma_k, h\},
\end{align*}
and
\begin{align*}
N_e^+ &\cong \z\{2(1 -\sigma_i), \sigma_i-\sigma_j, \sigma_i-\sigma_k, 41-h\}, \\
N_{C_2}^+ &\cong \z\{2(1-\sigma_i), \sigma_i - \sigma_j, \sigma_i-\sigma_k, h\}, \\
N_{\langle a \rangle}^+ &\cong \z \{2(1-\sigma_a), \sigma_{a'}, \sigma_{a''}, h\}, \\
N_Q^+ &\cong \z\{\sigma_i, \sigma_j, \sigma_k, h\}.
\end{align*}

The $2$-fold intersections are as follows:
\begin{align*}
N_e^+ \cap N_{C_2}^+ &\cong \z \{2(1-\sigma_i), \sigma_i - \sigma_j, \sigma_i - \sigma_k\}, \\
N_e^+ \cap N_{\langle a \rangle}^+ &\cong \z \{ 2(1 - \sigma_a), \sigma_{a'} - \sigma_{a''}, h - 4 \sigma_{a'}\}, \\
N_e^+ \cap N_Q^+ &\cong \z \{ \sigma_i - \sigma_j, \sigma_i - \sigma_k, 4\sigma_i - h \}, \\
N_{C_2}^+ \cap N_{\langle a \rangle}^+ &\cong \z \{ 2(1 - \sigma_a), \sigma_{a'} - \sigma_{a''}, h\}, \\
N_{C_2}^+ \cap N_Q^+ &\cong \z\{ \sigma_i - \sigma_j, \sigma_i - \sigma_k, h\}, \\
N_{\langle a \rangle}^+ \cap N_{\langle a' \rangle}^+ &\cong \z \{ 2(1 - \sigma_a - \sigma_{a'}), \sigma_{a''}, h\}, \\
N_{\langle a \rangle}^+ \cap N_Q^+ &\cong \z \{ \sigma_{a'}, \sigma_{a''}, h\}.
\end{align*}

The $3$-fold intersections are as follows:
\begin{align*}
N_e^+ \cap N_{C_2}^+ \cap N_{\langle a \rangle}^+ &\cong \z \{ 2(1 - \sigma_a), \sigma_{a'} - \sigma_{a''}\}, \\
N_e^+ \cap N_{C_2}^+ \cap N_Q^+ &\cong \z \{ \sigma_i - \sigma_j, \sigma_i - \sigma_k\}, \\
N_e^+ \cap N_{\langle a \rangle}^+ \cap N_Q^+ &\cong \z \{\sigma_{a'} - \sigma_{a''}, h - 4 \sigma_{a'} \}, \\
N_e^+ \cap N_{\langle a \rangle}^+ \cap N_{\langle a' \rangle}^+ &\cong \z \{2(1 - \sigma_a - \sigma_{a'} - \sigma_{a''}), h-4\sigma_{a''}\}, \\
N_{C_2}^+ \cap N_{\langle a \rangle}^+ \cap N_Q^+ &\cong \z\{\sigma_{a'}-\sigma_{a''},h\}, \\
N_{C_2}^+ \cap N_{\langle a \rangle}^+ \cap N_{\langle a' \rangle}^+ &\cong \z \{ 2(1-\sigma_a - \sigma_{a'} - \sigma_{a''}), h\}, \\
N_{\langle a \rangle}^+ \cap N_{\langle a' \rangle}^+ \cap N_Q^+ &\cong \z \{\sigma_{a''}, h\}, \\
N_{\langle a \rangle}^+ \cap N_{\langle a' \rangle}^+ \cap N_{\langle a'' \rangle}^+ &\cong \z\{2(1-\sigma_a-\sigma_{a'}-\sigma_{a''}),h\},
\end{align*}

The $4$-fold intersections are as follows:
\begin{align*}
N_e^+ \cap N_{C_2}^+ \cap N_{\langle a \rangle}^+ \cap N_Q^+ &\cong \z \{ \sigma_{a'} - \sigma_{a''} \}, \\
N_e^+ \cap N_{C_2}^+ \cap N_{\langle a \rangle}^+ \cap N_{\langle a' \rangle}^+ &\cong \z \{ 2(1 - \sigma_a - \sigma_{a'} - \sigma_{a''}\}, \\
N_e^+ \cap N_{\langle a \rangle}^+ \cap N_{\langle a' \rangle}^+ \cap N_{\langle a'' \rangle}^+ &\cong \z \{ 2(1 - \sigma_a - \sigma_{a'} - \sigma_{a''})\}, \\
N_{C_2}^+ \cap N_{\langle a \rangle}^+ \cap N_{\langle a' \rangle}^+ \cap N_Q^+ &\cong \z\{ h\}, \\
N_{C_2}^+ \cap N_{\langle a \rangle}^+ \cap N_{\langle a' \rangle}^+ \cap N_{\langle a'' \rangle}^+ &\cong \z \{ 2(1-\sigma_a-\sigma_{a'} - \sigma_{a''}), h\}, \\
N_{\langle a \rangle}^+ \cap N_{\langle a' \rangle}^+ \cap N_{\langle a'' \rangle}^+ \cap N_Q^+ &\cong \z \{h\}.
\end{align*}

The $5$-fold intersections are as follows:
\begin{align*}
N_e^+ \cap N_{C_2}^+ \cap N_{\langle a \rangle}^+ \cap N_{\langle a' \rangle}^+ \cap N_Q^+ &\cong \{0\}, \\
N_e^+ \cap N_{C_2}^+ \cap N_{\langle a \rangle}^+ \cap N_{\langle a' \rangle}^+ \cap N_{\langle a'' \rangle}^+ &\cong \z\{2(1 - \sigma_a - \sigma_{a'} - \sigma_{a''})\}, \\
N_e^+ \cap N_{\langle a \rangle}^+ \cap N_{\langle a' \rangle}^+ \cap N_{\langle a'' \rangle}^+ \cap N_Q^+ &\cong \{0\}, \\
N_{C_2}^+ \cap N_{\langle a \rangle}^+ \cap N_{\langle a' \rangle}^+ \cap N_{\langle a'' \rangle}^+ \cap N_Q^+ &\cong \z\{h\}.
\end{align*}

For completeness, the unique $6$-fold intersection is
$$N_e^+ \cap N_{C_2}^+ \cap N_{\langle a \rangle}^+ \cap N_{\langle a' \rangle}^+ \cap N_{\langle a'' \rangle}^+ \cap N_Q^+ \cong \{0\}.$$

\begin{prop}
With $a$, $a'$, $a''$ as above, we have
\[
\rk \pi_V^K(S^0) = 
\begin{cases}
	6 \quad & \text{ if } V=0, \\
	5 \quad & \text{ if } V \in ( \z\{h\} \cup \z\{2(1-\sigma_a-\sigma_{a'}-\sigma_{a''})\}) \setminus \{0\}, \\
	4 \quad & \text{ if } V \in (\z\{\sigma_{a'}-\sigma_{a''}\} \cup \z\{2(1-\sigma_a-\sigma_{a'}-\sigma_{a''}),h\}) \\
			 & \quad \quad \quad \setminus (\z\{h\} \cup \z\{2(1-\sigma_a-\sigma_{a'}-\sigma_{a''})\}), \\
	3 \quad & \text{ if } V \in (\bigcup_{H, H', H'' } N_H^+ \cap N_{H'}^+ \cap N_{H''}^+) \setminus (\bigcup_{H, H', H'', H'''} \bigcap_{L \in \{H, H', H'', H'''\}} N_L^+), \\
	2 \quad & \text{ if } V \in (\bigcup_{H \neq H'} N_H^+ \cap N_{H'}^+ ) \setminus (\bigcup_{H,H',H''} N_H^+ \cap N_{H'}^+ \cap N_{H''}^+), \\
	1 \quad & \text{ if } V \in (\bigcup_H N_H^+) \setminus (\bigcup_{H \neq H'} N_H^+ \cap N_{H'}^+), \\
	0 \quad & \text{ otherwise.}
\end{cases}
\]
\end{prop}

\bibliographystyle{alpha}
\bibliography{GSerreFiniteBibliography}

\end{document}